\documentclass[reqno]{amsart} \usepackage{amscd}
\usepackage{epsf}
\newtheorem{theorem}{Theorem}[section]

\newtheorem{lemma}[theorem]{Lemma}

\theoremstyle{remark} 

\begin{document}

\title[Geometric Quantization of the vortices]{The Dimension of the Hilbert space of geometric quantization of vortices on a Riemann surface}

\author{Rukmini Dey*, Saibal Ganguli**}
\address{* I.C.T.S.-T.I.F.R.,\\
Bengaluru, India.\\
email: rukmini@icts.res.in, rukmini.dey@gmail.com }

\address{** Harish Chandra Research Institute, Allahabad \\
email: saibalgan@gmail.com}

\maketitle

\begin{abstract}
In this article we  calculate  the dimension of the Hilbert space of K\"{a}hler quantization of the moduli space of vortices on a Riemann surface. This dimension  is given by the holomorphic Euler characteristic of the quantum line bundle.
\end{abstract}

\vskip 5mm


\vskip 2mm 

Mathematics  Subject Classification (MSC2010) :  mathematical physics ( 81V99), differential geometry (58A32).

\section{{\bf Introduction}}
Given a symplectic manifold $(M, \omega)$, with $\omega $ integral (i.e.  its cohomology class is in $H^2(M, {\mathbb Z})$), 
geometric pre-quantization is the construction of a Hermitian line bundle with a connection (called the pre-quantum line bundle) whose curvature $\rho$ is proportional to the symplectic form. This is always possible as long as $\omega$ is integral. This method of quantization  developed by Kostant and Souriau,  assigns to functions $f \in C^{\infty}(M)$, an operator, $\hat{f} =-i \nabla_{X_f} +   f$ acting on the Hilbert space of square integrable sections of $L$ (the wave functions). Here $\nabla = d - i \theta$ where locally $\omega = d \theta$ and $X_f$ is defined by $\omega(X_f, \cdot) = - df(\cdot)$. We have taken $ \hbar=1$.  The general reference for this is Woodhouse, ~\cite{W}.  

This assignment has the property that  that the Poisson bracket (induced by  the symplectic form), namely, 
$\{ f_1, f_2 \}_{PB} $ corresponds  to an operator proportional to the commutator $[\hat{f}_1, \hat{f}_2]$ for any two functions $f_1 , f_2$. 

The Hilbert space of pre-quantization is usually too huge for most purposes. 
Geometric quantization involves  construction of a polarization of the symplectic manifold such that we now take polarized sections of the line bundle, yielding a finite dimensional Hilbert space in most cases. However, $\hat{f}$ does not 
map the polarized Hilbert space to the polarized Hilbert space  in general. Thus only a few observables from the set of all  $f \in C^{\infty}(M)$ are quantizable.

We mention the general theory behind holomorphic quantization.

When $M$ is a compact  K\"{a}hler manifold,  $\omega$ is an integral  K\"{a}hler form, and $L$ the prequantum line bundle,   one can take holomorphic square integrable sections of $L^{\otimes \mu}$ for $\mu \in {\mathbb Z}$ large
enough as the Hilbert space of the quantization. (Recall $L$ might not have any global sections so we need to have high enough tensor product, hence we need $L^{\otimes \mu}$).   $L^{\otimes \mu}$ has curvature proportional to $\mu \omega$.
Let $\{ z_1, ....z_n\}$ be local co-ordinates on the K\"{a}hler manifold. 
The polarized sections are given by $\nabla_{\bar{\zeta}} \psi =0$ where $\nabla $ is the covariant dervative w.r.t. a
connection $\theta$ (i.e. $\mu  \omega = d \theta$ locally) and $\bar{\zeta} = \{ \frac{\partial}{\partial \bar{z_1}}, ....,  
 \frac{\partial}{\partial \bar{z_n}} \}$ and $\psi$ is a polarized section of $L^{\otimes \mu}$. Let $\theta = \sum_{i=1}^n f_i d z_i + \sum_{i=1}^n g_i d \bar{z}_i$. Then the equation for $\psi$ is
$\frac{\partial \psi}{\partial \bar{z}_i }- i g_i \psi =0$, $i=1,...,n$. One can show that the most general solution of this set of pde's is $\psi = f_0(z, \bar{z}) h(z)$, where $h(z)$ is a holomorphic section  of $L^{\otimes \mu}$ and $f_0(z, \bar{z})$ satisfies
the equation $\frac{\partial f_0}{\partial \bar{z}_i} - i g_i f_0 =0$, $i=1,...,n$. Since $f_0$ is completely determined by $\theta$,  there is a $1$ to $1$ correspondence between 
$\psi$ and $h$ i.e. polarized sections and holomorphic sections.

\vspace{.2in}

Let $(X, \omega)$ be a compact  integral K\"{a}hler manifold such that $L$ is a holomorphic  line bundle whose curvature is 
proportional to the K\"{a}hler form $\omega$. Let $L^{\otimes \mu}$ be the quantum line bundle for any $\mu \geq \mu_0 \in {\mathbb Z}^+$. It is a positive line bundle and $\mu_0$ has been chosen 
such that $H^q(X, L^{\otimes \mu}) =0$ for all $q >0$. By Kodaira's theorem, such a $\mu_0$ exists (Theorem B, \cite{GH}, page 159, taking $E$ to be trivial). Let the Hilbert space of quantization be the space of holomorphic sections
of $L^{\otimes \mu}$. It should be noted that since the $X$ is compact square integrability of holomorphic sections follows automatically.

The dimension of the Hilbert space of geometric quantization as defined above for any compact K\"{a}hler manifold with integral K\"{a}hler form $ \mu \omega$
is given by the holomorphic Euler characteristic of $L^{\otimes \mu}$ for all  $\mu \geq  \mu_0 \in {\mathbb Z}^+$, where $\mu_0$ is as above. This well known fact follows trivially since the holomorphic Euler char is the alternating sum of dimensions of the cohomology of $L^{\otimes \mu}$.
Since only $H^0$ survives, we get that the dimension of the space of global sections of $L^{\otimes \mu}$ is precisely the holomorphic Euler charateristic.

\vspace{.2in}

The main aim of this article is to calculate the dimension  of the Hilbert space of geometric quantization of the vortex moduli space.

The vortex equations are as follows. Let $\Sigma$ be a compact Riemann 
surface and let $\omega = h^2 dz \wedge d \bar{z}$ be the purely imaginary 
volume form on it, (i.e. $h$ is real). 
Let  $A$ be a unitary connection on a principal $U(1)$ bundle $P$
 i.e. $A$ is a purely imaginary valued one form i.e. $A = A^{(1,0)} + A^{(0,1)}$ such that $A^{(1,0)} = -\overline{A^{(0,1)}}$.   
Let ${\mathcal L}$ be a complex line bundle associated to $P$.
  Let $\Phi$ be a section of ${\mathcal L}$, i.e. 
$\Phi \in  \Gamma(\Sigma , {\mathcal L})$. 
There is a Hermitian metric $H$ on ${\mathcal L}$, i.e. the inner product $<\Phi_1, \Phi_2>_H = \Phi_1 H \bar{\Phi}_2$ is a smooth function on $\Sigma$. (Here $H$ is real). 

The pair $(A, \Phi)$ will be said to satisfy the vortex equations if

\hspace{1in} $(1)$ $ \rm{\;\;\;\;\;}$ $ F(A) = \frac{1}{2}(1-|\Phi|^2_H) \omega,$

\hspace{1in} $(2)$ $\rm{\;\;\;\;\;}$ $\bar{\partial}_A \Phi = 0,$

where $F(A)$ is the curvature of the connection $A$ and $d_A = \partial_A + \bar{\partial}_A $ is the decomposition of the covariant derivative operator
into $(1,0)$ and $(0,1)$ pieces.  
Let ${\mathcal S}$ be the space of solutions to $(1)$ and $(2)$.
 There is a gauge group $G$ acting on the space of $(A, \Phi)$ which leaves the equations invariant. Locally the gauge group   looks like ${\rm Maps} (\Sigma, U(1)).$ If $g$ is an $U(1)$ gauge transformation then 
$(A_1, \Phi_1)$ and $(A_2, \Phi_2)$ are gauge equivalent if 
$A_2 = g^{-1}dg + A_1 $ and $\Phi_2 = g^{-1} \Phi_1$. 
  Taking the quotient by the gauge group of ${\mathcal S}$ gives  the moduli 
space of solutions to these 
equations and is denoted by ${\mathcal M}$.
The vortex moduli space for $N$ vortices on a closed  Riemann surface $\Sigma$,  is well 
known to be the $Symm^N (\Sigma)$, where $Symm$ stands for the symmetric product.
 This is a compact  K\"{a}hler manifold, with the K\"{a}hler form given by Manton and Nasir as follows:
 $$\omega_{MN} = \frac{i}{2} \sum_{i,j=1}^N (\Omega(z_i) \delta_{ij} + 2 \frac{\partial b_i}{\partial \bar{z}_j}) d z_i \wedge d \bar{z}_i$$
 where $b_i$'s are terms appearing in the expansion of $log|\Phi|^2$ around a typical zero $z_i$ of $\Phi$ -- their derivatives  survive in the Samol's metric on the moduli space, Manton and Nasir, ~\cite{MN}.

  This form  is integral if we normalize  the volume $A$ of the Riemann 
surface to be in $4 \pi {\mathbb Z}^{+}$, ~\cite{MN}. In this case, if we take $A$ to be large enough, we show that the quantum bundle is a holomorphic bundle  $L$  whose 
curvature is proportional to $ \omega_{MN}$ and from the above  its holomorphic sections are  in  to one correspondence with the polarized sections constituting  the Hilbert space of quantization. 
In ~\cite{D}, ~\cite{DM} we had carried out this programme of geometric quantization of the vortex moduli space and had shown that the quantum line bundle is a Quillen determinant line bundle.

In this article, we calculate the dimension of this Hilbert space when the Riemann surface is a compact Riemann surface of genus $g$ (with $\mu =1$). In all case, if $A$ is sufficiently large,  (i.e. $\frac{A}{4 \pi}  > max ( N, g-1)$ ) it is the holomorphic Euler characteristc of the quantum bundle $L$.   The proof  is a simple application of the Kodaira Vanishing theorem, ~\cite{GH}.

The result also tallies with Romao's result for vortices on sphere, ~\cite{R},  as we explain in Remark $2$.

We must mention that in the arxiv there is a paper by Eriksson and
Romao ~\cite{ER}  where the Kahler quantization of the moduli space of vortices
is carried out using Deligne's point of view, and our result follows from
theirs. Our work is independent and uses the properties of the Manton-Nasir form.

\section{{\bf The dimension of the Hilbert Space of the vortices on a Riemann surface of genus $g$.}}

Let $X$ be the $N$-th symmetric product of a compact  Riemann surface $\Sigma$ of genus $g$. The volume of the Riemann surface has been normalized to $A= 4 \pi k$, $k \in {\mathbb Z}^+$. 

By Macdonald, ~\cite{Mac}, the cohomology ring of $X$ is given as follows. We follow the exposition given in ~\cite{MN}.

Let    $\alpha_i$, $i = 1,...,2g$ be one forms of $\Sigma$ whose cohomology classes are  the generators of $H^{1}(\Sigma, {\mathbb Z})$ and let the cohomology class of $\beta$ be the 
generator of $H^2(\Sigma, {\mathbb Z})$.  Note that $\beta $ is normalized such that $\int_{\Sigma} \beta =1$ and is of type $(1,1)$. The ring structure of $H^*(\Sigma, {\mathbb Z}) $ can be described as 

$[\alpha_i] [\alpha_j] =0$, $ i \neq j \pm g$;

$[\alpha_i] [\alpha_{i + g}] = - [\alpha_{i + g}][ \alpha_i] = [\beta]$,  $ 1 \leq i \leq g$,
where square brackets denote cohomolgy classes.
Let ${\alpha}_{ik}$ and ${\beta}_k$  are forms in $N$ product of $\Sigma$ (namely ${\Sigma}^{N}$) such that

Let $[\alpha_{ik}] = 1 \otimes ...\otimes 1 \otimes [\alpha_i] \otimes 1 \otimes...\otimes 1 \in H^1({\Sigma}^{N}, {\mathbb Z})$.

$[\beta_k] = 1 \otimes ...\otimes1 \otimes[\beta] \otimes 1 \otimes...\otimes 1 \in H^2({\Sigma}^{N}, {\mathbb Z})$

where $[\alpha_i]$ and $[\beta]$ are in the $k^th$ place. We take $\beta_k={p_k}^{*}(\beta)$ where $p_k:{\Sigma}^{N} \rightarrow \Sigma$ is the projection to the $k^{th}$ component.  Also, $\alpha_{ik} = p_k^*(\alpha_i)$. 

Then $H^*({\Sigma}^{N}, {\mathbb Z})$ is generated by $[\alpha_{ik}]$ and $[\beta_k]$ ( $1 \leq i \leq 2g, 1 \leq k \leq N$) with the following relations:

$[\alpha_{ik}][\alpha_{jk}] =0, $ $i \neq j \pm g$

$[\alpha_{ik}][\alpha_{i + g, k}] = - [\alpha_{i + g, k}] [\alpha_{ik}] = [\beta_k] $, $ 1 \leq i \leq g$

$[\alpha_{ik}] [\alpha_{jl}] = - [\alpha_{jl}] [\alpha_{ik}$], $k \neq l$.

Now define the following

$\zeta_i = \alpha_{i1} + ...+ \alpha_{iN}$, $1 \leq i \leq 2g.$

$\eta = \beta_1 + ...+ \beta_N$

Let $\zeta^{\prime}_i = \zeta_{i+g}$,  $ 1 \leq i \leq g$ and $\sigma_i = \zeta_i \zeta_i^{\prime}$.

These   forms are $S_N$-invariant and their cohomology classes  generate the cohomology of $X$.

\begin{lemma}
$\eta$ is positive. 
\end{lemma}
\begin{proof}
The form $\eta$ can be realized as a $G$- invariant form  $G = S_N$, the symmetric group,   given by
 \begin{equation}
  \eta= \beta_1 + \ldots + \beta_N
 \end{equation}
where $\beta_j$ is the normalized volume form on $\Sigma$, i.e. $\int_{\Sigma} \beta_j = 1$.  
In other words, 
$\eta = \frac{1}{A} \omega_1 = \frac{i}{2A} \sum_{j=1}^{N} \Omega(z_j) d z_j \wedge d \bar{z_j}$.
 (notation as in ~\cite{MN}). 
By definition (see ~\cite{GH}, page 29) this is a positive form since $ \Omega(z_j) >0$ (see (2.1) in ~\cite{MN}) and $A >0$.

\end{proof}

{\bf Remark 1:}
 It should be noted that in our paper we have denoted  $\omega_{MN}$,  $\eta$ and $\sigma_i$ as forms rather than their cohomology classes. We have used parentheses to denote the cohomology classes.
 This distinction has not been made in the papers we have cited ~\cite{Mac}, ~\cite{MN}.
 
\begin{theorem}
The dimension of the Hilbert space of geometric quantization for $N$-vortices on a compact Riemann surface of genus $g$  is given by the holomorphic Euler characteristic of the 
quantum line bundle $L$ on the $X=Symm^N( \Sigma)$ which has curvature proportional to the Manton-Nasir form, $\omega_{MN}$, a K\"{a}hler form  
( which is integral if the volume of the surface is $ A= 4 \pi k$, $k \in {\mathbb Z}^+$ ). 
The dimension of the Hilbert space is $  {  k \choose N} $, where $k = \frac{A}{4 \pi}$ if $\frac{A}{4\pi} > max(N, g-1)$ .
\end{theorem}
  
\begin{proof}

Note that the complex dimension of $X$ is $N$.

$\omega_{MN}$ is integral if $A =  4 \pi k$, $k \in {\mathbb Z}^+$, ~\cite{MN}. Thus there is a quantum holomorphic bundle, $L$ whose curvature is proportional to $\omega_{MN}$.

{\bf Step 1} We prove that the dimension of the Hilbert space for a Riemann surface of genus $g$  is  the holomorphic Euler characteristic for the quantum bundle $L$  whose curvature is $\omega_{MN}$.

By ~\cite{MN} the Manton Nasir form is given by
\begin{equation}
 \omega_{MN}= (A-4\pi N)\eta +4\pi(\sigma_1 +\ldots+ \sigma_g) + \text{exact terms}.
\end{equation}
Let $[\omega_{MN}]$ signify cohomology class of $\omega_{MN}$. Then
\begin{equation}
\frac {[\omega_{MN}]}{4\pi}=  (\frac{A}{4\pi}-N)[\eta] +( [\sigma_1]+ \ldots +[\sigma_g)
\end{equation}
It  is an integral form if $\frac{A}{4\pi}$ is integral.
 
 We also have from ~\cite{MN} 
 \begin{equation}
  c_1(TX )=(N-g+1)[\eta] -([\sigma_1] + \ldots +[\sigma_g])
 \end{equation}
where $X$ is the symmetric product.

From the above two equations we have
\begin{equation}
\frac {[\omega_{MN}]}{4\pi} + c_1(TX)=(\frac{A}{4\pi} - g+1)[\eta]
\end{equation}
 Recall that  $L$ is the line bundle representing  the first term in left hand side of the above equation. Let  $K^{*}$ be the anti-canonical bundle. Then the tensor product  $L^{\prime} = L \otimes K^*$ is a positive bundle  since $(\frac{A}{4\pi}-g+1) >0$ and  $\eta$ is positive by Lemma $(2.1)$. Then $L = L^{\prime} \otimes K$ has,  by Kodiara Vanishing theorem, zero higher cohomology,  ~\cite{GH}, page 154. Hence the
 global sections are given by the holomorphic Euler characteristic for $L$.

{\bf Step 2:} 
The  Hirzebruch-Riemann Roch theorem implies  that the holomorphic Euler characteristic $H.E.(X,L)$  of a $L$  is computable in terms of the Chern class $c_1(L) $ (the class of $\omega_{MN}$) and the Todd polynomial  in the Chern classes of the holomorphic tangent bundle of $X$. In fact,  $H.E.(X, L) = \int_X ch(L) td(X) $ where $ch(L) $ is the Chern character of $L$ and $td(X)$ is the Todd class of the tangent bundle of $X$.

However, in our case, Macdonald has already calculated the holomorphic Euler char using Hirzebruch-Riemann Roch formula in ~\cite{Mac} and we can invoke his results to calculate the holomorphic Euler characteristic of $L$.

For a vector bundle $E$ on $X$ its holomorphic Euler characteristic  is given by
 \begin{equation}
 H.E.(E)=\kappa_N \{ (e^{\delta_1}+ \ldots + e^{\delta_n} )e^{-[\eta] g}{(\frac{[\eta]}{1-e^{-[\eta]}})}^{n+1} \}
 \end{equation}
 where product of $1+\delta_i$ determine the Chern classes of the $E$ and $\kappa_N$ signifies we have to only consider highest degree forms,  ~\cite{Mac}.

  Now Chern class of the quantum bundle $L$ is
  \begin{equation}
  c_1(L)=(\frac{A}{4\pi}-N)[\eta] +[\sigma_1]  +\ldots + [\sigma_g]
  \end{equation}
 Substituting in the above equation
 \begin{equation}
  H.E.(L)= \kappa_N \{ e^{(\frac{A}{4\pi}-N)[\eta] +[\sigma_1]  +\ldots + [\sigma_g]}e^{-[\eta] g}{(\frac{[\eta]}{1-e^{-[\eta]}})}^{N+1} \}
 \end{equation}
By $(5.4)$ in ~\cite{Mac}, we can replace each $\sigma_i$ by $\eta$ (as in page $334$ below equation 14.10~\cite{Mac}), to get 
\begin{eqnarray*}
H.E.(L) &=& \kappa_N  \{e^{(\frac{A}{4\pi}-N)[\eta]}{(\frac{[\eta]}{1-e^{-[\eta]}})}^{N+1})\\
&=& Res_{[\eta]}(e^{(\frac{A}{4\pi}-N)[\eta]}{(\frac{1}{1-e^{-[\eta]}})}^{N+1}) 
\end{eqnarray*}
Substituting $\varepsilon =1-e^{-[\eta]} $, $d \varepsilon = (1 - \varepsilon) d [\eta ]$ and letting $K_0 = \frac{A}{4 \pi} -N > 0$, we get 

$H.E..(L) = Res_{\varepsilon} \frac{(1-\varepsilon)^{-K_0- 1}}{\varepsilon^{N+1}}$.
 
Thus $H.E.(L) = $  the coefficient of $\varepsilon^N $ in 
$(1 - \varepsilon)^{-K_0 -1}$
 which is ${ K_0 + N  \choose N }$ or $ {  k \choose N}$ where $k = \frac{A}{4 \pi} $. This number is always $\geq 1$.

\end{proof}

{\bf Remark 2:}
 In \cite{R} Romao  has calculated the dimension of the Hilbert space for vortices on the sphere to be ${ N+l \choose l} $ where $l$ is Chern class of the quantum bundle. Here we identify
 the  second cohomology of $CP^N$ (the vortex moduli space for a sphere) with the integers. Since
 ${ N+l \choose l} $ is equal to $ {N+l \choose N} $ and  since Manton-Nasir form in the  genus zero case is $(\frac{A}{4 \pi}-N) \eta$, therefore  $l= \frac{A}{4 \pi}-N $ and the result tallies with our result .


\begin{thebibliography}{99}


\bibitem{D} Dey, R.,
Erratum: Geometric prequantization of the moduli space of the vortex equations 
on a Riemann surface,  Journal of Mathematical  
Phys. 50, 119901 (2009)  


\bibitem{DM} Dey, R.,  Mathai, V.:  Holomorphic Quillen determinant bundle on integral compact 
 K\"{a}hler manifolds;
Quart. J. Math. 64 (2013), 785-794, Quillen Memorial Issue;  arXiv:1202.5213v3


\bibitem{ER} Eriksson, D., Romao, N. :  Kahler Quantization of Vortex Moduli; arxiv: 1612.08505

\bibitem{GH} Griffiths, P., Harris, J. : Principles of Algebraic Geometry; Wiley Classics Library Edition, 1994.
  

\bibitem{MN} N.S. Manton, S.M. Nasir: Volume of vortex moduli spaces;  
Comm. Math. Phys.  199,  no. 3, 591--604 (1998)


\bibitem{Mac} Macdonald, I.G.: Symmetric products of an algebraic curve; Topology, Vol 1, page 319-343, 1962.






\bibitem{R}  Romao, N.M.: Quantum Chern-Simons vortices on a sphere;    J.Math.Phys. 42 (2001) 3445-3469


\bibitem{W} Woodhouse, N.M.J. :  Geometric quantization; The Clarendon press, 
Oxford University Press, New York (1992).
\end{thebibliography}
\end{document}